\documentclass[12pt]{amsart}
\usepackage[dvipdfmx]{graphicx}
\usepackage{color}
\usepackage{amsmath,amssymb,amsthm}
\usepackage{amscd}
\usepackage{comment}
\usepackage{overpic}
\usepackage{tabularx}
\usepackage[dvipdfm,left=30truemm,right=30truemm,top=30truemm,bottom=30truemm]{geometry}

\usepackage{bm}

\theoremstyle{plain} 
\newtheorem{theorem}{Theorem}[section] 

\newtheorem{corollary}[theorem]{Corollary}
\newtheorem{proposition}[theorem]{Proposition}

\theoremstyle{definition}

\newtheorem{remark}[theorem]{Remark}

\makeatletter
\def\Bline{%
\noalign{\ifnum0=`}\fi\hrule \@height 1pt \futurelet
\reserved@a\@xhline}
\@addtoreset{equation}{section}

\makeatother

\allowdisplaybreaks[3]

\newcommand{\R}{\mathbb{R}}

\renewcommand{\phi}{\varphi}
\renewcommand{\epsilon}{\varepsilon}

\renewcommand{\geq}{\geqslant}

\newcommand{\ma}{\mathcal{A}}

\begin{document}

\title[Height functions on singular surfaces with $H_k$ singularities]
{Height functions on singular surfaces parameterized by
smooth maps $\mathcal{A}$-equivalent to $H_k$}

\author[M.~Hasegawa]
{Masaru Hasegawa}

\address[Masaru Hasegawa]{%
Department of Information Science, Center for Liberal Arts and Sciences, Iwate Medical University,
1-1-1 Idaidori, Yahaba-cho, Shiwa-gun, Iwate 028-3694, Japan.}
\email{mhase@iwate-med.ac.jp}
%
%

\thanks{The research underlying this work was supported by FAPESP post-doctoral grant 2013/02543-1 during the author's post-doctoral period at ICMC-USP}


\begin{abstract}
We study singularities of height functions on singular surfaces in $\mathbb{R}^3$ parameterized by smooth map-germs $\mathcal{A}$-equivalent to $H_k$ in Mond's classification, and the versality of the family of the height functions. 
We also study relations of the singularities of the height functions with the parabolic locus of these singular surfaces.
\end{abstract}


\maketitle


\section{Introduction}
\label{sec:intro}

Two map-germs $f, g\colon(\R^2,\bm{0})\to(\R^3,\bm{0})$ are said to be \textit{$\mathcal{A}$-equivalent} if there exist germs of diffeomorphisms $\phi\colon(\R^2,\bm{0})\to(\R^2,\bm{0})$ and $\Phi\colon(\R^3,\bm{0})\to(\R^3,\bm{0})$ such that $g = \Phi \circ f \circ \phi^{-1}$. 
In \cite{Mond1985}, D.~Mond classified smooth map-germs $(\R^2,\bm{0})\to(\R^3,\bm{0})$ under $\mathcal{A}$-equivalence and gave a list (Table \ref{tab:A-simple}) of normal forms of the map-germs.

\begin{table}[ht!]
\caption{Classes of $\mathcal{A}$-simple map-germs.}
\centering
\begin{tabular}{ccc}
\Bline
Name & Normal form & $\mathcal{A}$-codim.\\\hline
Immersion & $(x,y,0)$ & $0$ \\
Whitney umbrella ($S_0$) & $(x,y^2,x y)$ & $2$ \\
$S_k^\pm$ & $(x,y^2,y^3 \pm x^{k+1} v)$, $k\geq1$ & $k+2$\\
$B_k^\pm$ & $(x,y^2,x^2y \pm y^{2k+1})$, $k\geq2$ & $k+2$\\
$C_k^\pm$ & $(x,y^2,x y^3 \pm x^k y)$, $k\geq3$ & $k+2$\\
$F_4$ & $(x,y^2,x^3y + y^5)$ & $6$\\
$H_k$ & $(x,x y + y^{3k-1},y^3)$, $k\geq2$ & $k+2$\\\Bline
\end{tabular}

 (When $k$ is even, $S_k^+$ is equivalent to $S_k^-$, and $C_k^+$ to $C_k^-$.)
\label{tab:A-simple}
\end{table}

The contact of surfaces with planes can be measured by $\mathcal{K}$-singularities of height functions on the surfaces in the normal directions of the planes. 
Two map-germs $f, g\colon(\R^2,\bm{0})\to(\R,0)$ are said to be \textit{$\mathcal{K}$-equivalent} if there exist a germ of diffeomorphism $\phi\colon(\R^2,\bm{0})\to(\R^2,\bm{0})$ and a function-germ $\lambda\colon(\R^2,\bm{0})\to\R$ with $\lambda(\bm{0})\ne0$ such that $g(\bm{x}) = \lambda(\bm{x})f \circ \phi^{-1}(\bm{x})$. 
In this paper, we use $\mathcal{K}$-singularities with normal forms as follows:
\[
A_k\ (\text{or}\ A_k^{\pm}) \colon x^2 \pm y^{k+1}, \quad
D_k\ (\text{or}\ D_k^{\pm}) \colon x^2 y \pm y^{k-1}\,(k\geq 4).
\]

It is well known that parameterized surfaces in $\R^3$ can have singularities of type Whitney umbrella (also called cross cap) and this singularity type is the only stable singularity of maps of $\R^2$ into
$\R^3$.
The differential geometry of Whitney umbrellas is studied in, for example, \cite{BW1998, FH2012, FH2013, GGS2000, HHNSUY2015, HHNUY2014,
Tari2007, West1995}. 
The differential geometry of surfaces in $\mathbb{R}^3$ with corank $1$ singularities has been studied, for example, in \cite{FH2025-2, FH2025, MN-B2015, O-SS2022, O-ST2015, Saji2018, Shimada2024}.

In a joint work \cite{FH2013} of T.~Fukui and the author, they study singularities of height functions on Whitney umbrellas in terms of the extended differential geometric properties via a blowing up obtained in \cite{FH2012}.
In \cite{FH2025-2, FH2025}, they extend the methods in \cite{FH2012,FH2013} to study differential geometry of singular surfaces parameterized by smooth map-germs $\mathcal{A}$-equivalent to one of $S_k$, $B_k$, $C_k$ and $F_4$. 
The $H_k$-case was not treated in \cite{FH2025-2,FH2025}, since their 2-jets are $\mathcal{A}$-equivalent to $(u,v^2,0)$, whereas the normal form of $H_k$ has the 2-jet $\mathcal{A}$-equivalent to $(u,uv,0)$. 
Thus, the methods used in \cite{FH2025-2,FH2025} cannot be applied directly to the $H_k$-case.
In \cite{O-ST2015}, R.~Oset Sinha and F.~Tari studied singular surfaces in $\R^3$ whose parametrizations have an $\mathcal{A}$-singularity of $\mathcal{A}_e$-codimension less than or equal to $3$, appearing in projections of smooth surfaces in $\R^4$ to $\R^3$ and including those with $H_2$ and $H_3$. 
They studied these surfaces via their contact with planes and investigated the singularities of their preparabolic sets and height functions.

In this paper, we study height functions on a singular surface $S$ in $\R^3$ parameterized by smooth map-germs $\mathcal{A}$-equivalent to $H_k$, which was not treated in \cite{FH2025}, for $k\geq 2$, together with the versality of the family of the height functions and their relations with the parabolic set of $S$.

The paper is organized as follows:
In Section 2, we introduce parameterizations of singular surfaces with corank $1$ singularities and several geometric notions for these surfaces.
In Section 3, we describe singularities of height functions on $S$ in terms of the geometric notions introduced in Section 2 (Theorem \ref{thm:height}).
Theorem \ref{thm:height} shows that the types of singularities of height functions on $S$ and the versality of their family do not depend on the degree of degeneracy $k$ of $H_k$.
In Section 4, we describe the singularities in terms of geometric properties of branches of the parabolic set on $S$ (Theorem \ref{thm:branch}).

\section{preliminaries}

\subsection{Parameterizations of singular surfaces}

Let $S$ be a singular surface parameterized by a smooth map-germ $(\R^2,\bm{0})\to(\R^3,\bm{0})$ of corank $1$.
We can make changes of coordinates in the source and rotations in the target, which do not change the geometry of $S$. 
Then the map-germ can be written in the form
\[
(u,v)\mapsto(u,y(u,v),z(u,v)),
\]
where $y, z \in \mathcal{M}_2^2$.
Here, $\mathcal{M}_2$ is the maximal ideal of the local ring of smooth function-germs $(\R^2,\bm{0})\to\R$.

To investigate the differential geometry of $S$, the special parameterizations (obtained by using changes of coordinates in the  source and rotations in the target) of $S$ are useful. 
Such parameterizations are obtained, for example, for Whitney umbrellas  in \cite{West1995} (see also \cite{FH2012}), and for cuspidal edges in \cite{MS2016}.

\begin{proposition}[\cite{FH2025}]
\label{prop:normal_form}
Let $f\colon(\R^2,\bm{0})\to(\R^3,\bm{0})$ be a map-germ of corank 1 at the origin. 
Then, after using rotations in the target and changes of coordinates in the source, we can reduce $f$ to the form
\begin{equation*}
\left(u,\, \frac12 v^2 + \sum_{i=2}^k \frac{b_i}{i!} u^i + O(u,v)^{k+1},\,  \frac12 a_{2,0} u^2 + \sum_{m=3}^k\sum_{i+j=m} \frac{a_{i,j}}{i!j!}u^i v^j + O(u,v)^{k+1}\right),
\end{equation*}
if $j^2 f(\bm{0})$ is $\ma$-equivalent to $(u,v^2,0)$, or
\begin{equation}
\label{eq:normal_form_2}
\left(u,\, u v + \sum_{i=3}^k \frac{b_i}{i!} v^i + O(u,v)^{k+1},\, \frac12 a_{2,0} u^2 + \sum_{m=3}^k\sum_{i+j=m} \frac{a_{i,j}}{i!j!} u^i v^j + O(u,v)^{k+1}\right), 
\end{equation}
if $j^2 f(\bm{0})$ is $\ma$-equivalent to $(u,u v,0)$, where $O(u,v)^n$ consists of the terms of degree geater than or equal to $n$.
\end{proposition}
\begin{proof}
The proof of the first assertion is given in \cite{FH2025}, so we will give the proof of the second assertion.
 
We may assume that
 \[
j^2f(\bm{0}) = \left(u,\, \frac12 b_{2,0} u^2 + b_{1,1} u v + \frac12 b_{0,2} v^2,\, \frac12 a_{2,0} u^2 + a_{1,1} u v + \frac12 a_{0,2} v^2 \right).
\]
Let $j^2 f(\bm{0})$ be $\mathcal{A}$-equivalent to $(u,u v,0)$.
Then we have 
\[
\begin{vmatrix}
b_{1,1} & b_{0,2}\\
a_{1,1} & a_{0,2}\\
\end{vmatrix} = 0,
\quad (a_{0,2},b_{0,2})=(0,0), \quad \mbox{and} \quad (a_{1,1},b_{1,1})\ne(0,0).
\]
Let $R$ be the orthogonal matrix defined by 
\[
R=
\begin{pmatrix}
1 & 0 \\
0 & R_1
\end{pmatrix}
\quad\text{where}\quad
R_1 = \frac1{\sqrt{a_{1,1}^2 + b_{1,1}^2}}
\begin{pmatrix}
b_{1,1} & a_{1,1}\\
-a_{1,1} & b_{1,1}
\end{pmatrix}.
\]
Then the 2-jet of $R g$ is 
\[
\left(u,\, \dfrac{a_{2,0}a_{1,1}+b_{2,0}b_{1,1}}{2\sqrt{a_{1,1}^2+b_{1,1}^2}}u^2 + \sqrt{a_{1,1}^2+b_{1,1}^2}u v,\, \dfrac{a_{2,0}b_{1,1}-a_{1,1}b_{2,0}}{2\sqrt{a_{1,1}^2+b_{1,1}^2}}u^2\right).
\]
Substituting $v$ by $c_{1,0} u + c_{0,1} v$ and choosing suitable coefficients $c_{1,0}$ and $c_{0,1}$, we show that the 2-jet of $R g$ is
\[
\left(u,\, u v,\, \dfrac{a_{2,0}b_{1,1}-a_{1,1}b_{2,0}}{2\sqrt{a_{1,1}^2+b_{1,1}^2}}u^2\right). 
\]
This proves the second assertion for $k=2$.
Assume that $k\geq 3$.
Substituting $v$ by $v + \sum_{i+j=2}^{k-1} c_{i,j} u^i v^j/(i!j!)$ transforms the second component of $j^k f(\bm{0})$ to 
\[
u v + \sum_{m=3}^k \sum_{i=1}^m\left(\dfrac{b_{i,m-i}}{i!(m-i)!} + \dfrac{c_{i-1,m-i}}{(i-1)!(m-i)!}\right) u^i v^{m-i},
\]
and we can choose $c_{i,j}$ so that each term $u^iv^j$ with $i+j=k$ and $i\neq0$ is zero, which proves the second assertion.
\end{proof}
A similar result to Proposition \ref{prop:normal_form} is shown in
\cite{MN-B2015}. 

\begin{proposition}
\label{prop:H_k}
Let $f$ be a smooth map-germ $f:(\R^2,\bm{0})\to(\R^3,\bm{0})$ given in the form \eqref{eq:normal_form_2}.
If $f$ is of finite $\mathcal{A}$-codimension and $a_{0,3}\ne0$, then $f$ is $\mathcal{A}$-equivalent to $H_k$ for some $k\geq 2$.
In particular, $f$ is $\mathcal{A}$-equivalent to $H_2$ if and only if $a_{0,3}\ne0$ and 
\begin{align}
\label{eq:H2}
4a_{0,5}a_{0,3}b_3 - 5a_{0,4}^2b_3 + 5a_{0,4}a_{0,3}b_4 + 10 a_{0,4}a_{1,2}b_3^2 - 4a_{0,3}^2b_5 - 10a_{1,2}a_{0,3}b_4b_3 \ne 0.
\end{align}
\end{proposition}
\begin{proof}
By a suitable left coordinate change, we can reduce the term of $u^{m+n}v^n$ $(m\geq 0, n\geq0)$ of the third component of $f$ to zero.
So $j^3 f(\bm{0})$ is reduced to
\[
\left(u,\, u v + \dfrac{b_3}6 v^3,\, \dfrac16(3a_{1,2}u v^2 + a_{0,3}v^3)\right).
\]
If $a_{0,3} \ne 0$, then by the right coordinate change
\[
(u,v) \mapsto \left(u,\, - \dfrac{a_{1,2}}{a_{0,3}} u + v + \dfrac{a_{1,2} b_3}{2 a_{0,3}} v^2\right)
\]
and the left coordinate change
\[
(x,y,z) \mapsto \left(x,\, y + \dfrac{a_{1,2}}{a_{0,3}} x^2 - \dfrac{a_{1,2}^3 b_3}{6 a_{0,3}^3} x^3 + \dfrac{b_3}{a_{0,3}}z,\, z + \dfrac{a_{1,2}^2}{2 a_{0,3}} x y + \dfrac{a_{1,2}^3}{6 a_{0,3}^2} x^3
\right),
\]
we reduce $j^3 f(\bm{0})$ to
\[
j^3 f(\bm{0}) = \left(u,\, u v,\, \dfrac{a_{0,3}}6 v^3
\right),
\]
where $(x,y,z)$ is the usual Cartesian coordinate system of $\R^3$.
Therefore, the first assertion follows from Mond's classification \cite[Theorem 4.2.1:2(a)]{Mond1985}.

We next determine when the $\mathcal{A}$-equivalence class obtained
above is $H_2$. 
Since $H_2$ is 5-determined \cite[Theorem 4.2.1:2(b)]{Mond1985}, $f$ is $\mathcal{A}$-equivalent to $H_2$ if and only if $j^5 f(\bm{0})$ is $\mathcal{A}$-equivalent to $(u,\,u v + v^5,\,v^3)$. 

By suitable right and left coordinate changes, as in the proof of the first assertion, we can reduce $j^5f(\bm{0})$ to
$$
\left(u,\, uv+\frac{C}{a_{0,3}}v^5,\, \frac{a_{0,3}}6v^3 \right),
$$
where
$$
C= 4a_{0,5}a_{0,3}b_3-5a_{0,4}^2b_3+5a_{0,4}a_{0,3}b_4+10a_{0,4}a_{1,2}b_3^2-4a_{0,3}^2b_5-10a_{1,2}a_{0,3}b_4b_3.
$$
Hence, $C\ne0$ is equivalent to \eqref{eq:H2}, and therefore $f$ is $\mathcal{A}$-equivalent to $H_2$ if and only if $a_{0,3}\ne0$ and (2.2) holds.
\end{proof}

\subsection{Geometry of singular surfaces}

Let $S$ be a singular surface parameterized by a smooth map-germ $f\colon(\R^2,\bm{0})\to(\R^3,\bm{0})$ of corank 1 at the origin $\bm{0}$. 
At the singular point $f(\bm{0})$, the tangent plane degenerates to a line, that is, the image of $df_{\bm{0}}$ is a line.  
We call this line the {\it tangent line}.
The plane passing through $f(\bm{0})$ perpendicular to the tangent line is called the {\it normal plane}.
 
There exists non-zero vector $\eta\in T_{\bm{0}}\R^2$ such that $df_{\bm{0}}(\eta)=0$.
We call $\eta$ a {\it null vector} (see \cite{KRSUY2005}). 
Let $j^2 f(\bm{0})$ be $\mathcal{A}$-equivalent to $(u,u v,0)$.
The plane passing through $f(\bm{0})$ spanned by $\xi f(\bm{0})$ and $\xi\eta f(\bm{0})$ is called the {\it principal plane},  where $\xi \in T_{\bm{0}}\R^2$ is a non-zero vector such that $\{\xi,\eta\}$ is linearly independent and $\zeta g$ is the directional derivative of a vector valued function $g$ along the direction $\zeta$.
The unit normal vector to the principal plane at the singular point $f(\bm{0})$ is called the \textit{principal normal vector}.

We remark that the definitions of these geometric ingredients are independent of the choice of coordinates in the source (see \cite{HHNUY2014}).
We also remark that the definition of the principal plane is different from that for $S$ parameterized by $f$ whose $2$-jet is $\mathcal{A}$-equivalent to $(u,v^2,0)$ (cf. \cite{FH2025}).

We consider the orthogonal projection of $f$ onto the normal plane.
The projection can be expressed as
\[
(\R^2,\bm{0})\to(\R^2,\bm{0}),\quad(u,v)\mapsto(p(u,v),q(u,v)).
\]
We consider the group $\mathcal{G} = \mathrm{GL}(2,\R) \times \mathrm{GL}(2,\R)$ which acts on $(j^2 p, j^2 q)$. 
The list of $\mathcal{G}$-orbits is given in Table \ref{tab:singular_point} (see, for example, \cite{Gibson1979}). 
We classify the singular points of $S$ on the basis of the $\mathcal{G}$-class of $(j^2 p, j^2 q)$ in Table \ref{tab:singular_point}.
From Proposition \ref{prop:normal_form}, if $j^2f(\bm{0})$ is $\mathcal{A}$-equivalent to $(u,v^2,0)$ then the singular point of $S$ is a hyperbolic, inflection or degenerate inflection point. 
On the other hand, if $j^2f(\bm{0})$ is $\mathcal{A}$-equivalent to $(u,u
v,0)$, then the singular point is either a parabolic or inflection point (see \cite{O-ST2015} for details). 

\begin{table}[ht!]
\centering
\caption{The classification of the singular points.}
\begin{tabular}{cc}
\Bline
$\mathcal{G}$-class & Name \\\hline
$(x^2,y^2)$ & hyperbolic point\\
$(x y,x^2-y^2)$ & elliptic point\\
$(x^2,x y)$ & parabolic point\\
$(x^2\pm y^2,0)$ & inflection point\\
$(x^2,0)$ & degenerate inflection point\\
$(0,0)$ & degenerate inflection point\\\Bline
\end{tabular}
\label{tab:singular_point}
\end{table}

 
One can easily show the following proposition:
\begin{proposition}
\label{prop:properties_of_normal_form} 
Assume that $f$ is given in the form \eqref{eq:normal_form_2}.
\begin{itemize}
\item 
The tangent line is the $x$-axis and the normal plane is the $y z$-plane. 

\item 
The null vector can be chosen as $\eta = \partial_v$ and the principal plane is the $x y$-plane. 

\item 
The principal normal vector is $\pm \partial_z$. 

\item 
The point $f(\bm{0})$ is an inflection $($resp. parabolic$)$ point if and only if $a_{2,0} = 0$ $($resp. $a_{2,0} \ne 0$$)$.
\end{itemize}
\end{proposition}

\begin{remark}
For the parameterization \eqref{eq:normal_form_2},
we can take $\eta=\partial_v$, and hence $\eta^3f(\bm{0})=(0,b_3,a_{0,3})$ in the normal plane. 
The sign of $a_{2,0}$ can be normalized without changing this vector.
Indeed, consider the source coordinate change $\phi(u,v)=(u,-v)$ and the rotation $T(x,y,z)=(x,-y,-z)$ of $\R^3$ by $\pi$ about the $x$-axis, and set $\widehat{f}=T\circ f\circ\phi$.
Then $\widehat{f}$ is again of the form \eqref{eq:normal_form_2}, with $\widehat{a}_{2,0}=-a_{2,0}$, while $\eta^3\widehat{f}(\bm{0})=(0,b_3,a_{0,3}) =\eta^3f(\bm{0})$.
Consequently, replacing $f$ by $\widehat{f}$ when $a_{2,0}<0$, we may assume that $a_{2,0}\geq0$ while keeping $\eta^3f(\bm{0})$ fixed (cf. \cite[Introduction]{HHNUY2014} and \cite[Theorem~3.1]{MS2016}). 
Note that this normalization of $a_{2,0}$ is not used in this paper, because the statements of theorems and their proofs do not require this normalization. 
\end{remark}

A regular plane curve in the parameter space transverse to $\eta$ at $(0,0)$ is called a {\it tangential curve}. 
Let $\gamma(t)$ be a parameterization of the tangential curve. 
Clearly, $f\circ\gamma$ is tangent to the tangent line of the singular surface. 

\begin{remark}
By using Proposition \ref{prop:properties_of_normal_form}, it is easily seen that the singular point $f(\bm{0})$ is a parabolic point if and only if $f\circ\gamma$ has exactly 2-point contact with the principal plane at $f(\bm{0})$.
On the other hand, $f(\bm{0})$ is an inflection point if and only if $f\circ\gamma$ has an inflectional tangent to the principal plane at $f(\bm{0})$, that is, $f\circ\gamma$ has at least 3-point contact with the principal plane there.
\end{remark}


\section{Singularities of height functions}

We define the family of functions on a surface $S$ parameterized by a smooth map-germ $f:(\R^2,\bm{0})\to(\R^3,\bm{0})$ by
\[
H:(\R^2 \times S^2,(\bm{0}, \bm{w}_0))\to \R, \quad H(u,v,\bm{w}) = \langle f(u,v),\,\bm{w} \rangle,
\]
where $S^2$ is the unit sphere in $\R^3$ and $\langle\cdot, \cdot \rangle$ denotes the Euclidean inner product in $\R^3$. 
We define the function $h_{\bm{w}}(u,v) = H(u,v,\bm{w})$, which is the {\it height function on $S$ along  $\bm{w}$}.
We regard $\bm{w} \in S^2$ as a unit vector in $\R^3$ and we write
$\bm{w} = (x,y,z)$.

The following theorem describes singularities of $h_{\bm{w}}$ on singular surfaces $S$ parameterized by smooth map-germs $\mathcal{A}$-equivalent to $H_k$, and $\mathcal{R}^+$-versality of the family $H$ of $h_{\bm{w}}$, in terms of the geometric notions introduced in Section~2. 
We do not recall here definitions of unfoldings and their $\mathcal{R}^+$- and $\mathcal{K}$-versality. 
See \cite{Arnold1986} for the definitions.
See also \cite{IRRT2015, Martinet1982, Wall1981}.

\begin{theorem}
\label{thm:height}
Let $S$ be a singular surface parameterized by a smooth map-germ $\mathcal{A}$-equivalent to $H_k$.
\begin{enumerate}
\item 
$h_{\bm{w}_0}$ has an $A_1$-singularity at $(0,0)$ if and only if $\bm{w}_0$ is in the normal plane but $\bm{w}_0$ is not the principal normal direction. 
When this is the case, $H$ is an $\mathcal{R}^+$-versal unfolding of $h_{\bm{w}_0}$. 

\item 
$h_{\bm{w}_0}$ has an $A_2$-singularity at $(0,0)$ if and only if $\bm{w}_0$ is the principal normal direction and the singular point of $S$ is not an inflection point. 
When this is the case, $H$ is not an $\mathcal{R}^+$-versal unfolding of $h_{\bm{w}_0}$. 

\item 
$h_{\bm{w}_0}$ does not have an $A_{\geq 3}$-singularity at $(0,0)$.

\item 
$h_{\bm{w}_0}$ has a $D_4$- or more degenerate singularity at $(0,0)$ if and only if $\bm{w}_0$ is the principal normal direction and the singular point of $S$ is an inflection point. 
When this is the case, $H$ is not an $\mathcal{R}^+$-versal unfolding.
\end{enumerate}
\end{theorem}

\begin{remark}
It is well known that an $A_3$-singularity of a height function on a regular surface captures a cusp of Gauss of the surface.
In contrast, Theorem \ref{thm:height}(3) shows that height functions on a singular surface with an $H_k$-singularity do not detect an analogue of the cusp of Gauss of a regular surface.
\end{remark}

To prove Theorem \ref{thm:height}, we give the following proposition:

\begin{proposition}
\label{prop:height}
Let $S$ be parameterized by $f$ in the form \eqref{eq:normal_form_2}.
\begin{enumerate}
\item 
$h_{\bm{w}_0}$ has an $A_1$-singularity at $(0,0)$ if and only if $\bm{w}_0$ is in the normal plane $($i.e., $x_0 = 0)$ but $\bm{w_0}$ is not the principal normal vector $($i.e., $\bm{w}_0 \ne \pm(0,0,1))$.
When this is the case, $H$ is an $\mathcal{R}^+$-versal unfolding of $h_{\bm{w}_0}$. 

\item 
$h_{\bm{w}_0}$ has an $A_2$-singularity at $(0,0)$ if and only if $\bm{w}_0$ is the principal normal vector, $f(\bm{0})$ is not an inflection point $($i.e., $a_{2,0} \ne 0)$ and $a_{0,3} \ne 0$. 
When this is the case, $H$ is not an $\mathcal{R}^+$-versal unfolding of $h_{\bm{w}_0}$.

\item 
$h_{\bm{w}_0}$ has an $A_3$-singularity at $(0,0)$ if and only if $\bm{w}_0$ is the principal normal vector, $f(\bm{0})$ is not an inflection point, $a_{0,3} = 0$, and $a_{0,4} a_{2,0} - 3a_{1,2}^2 \ne 0$.
When this is the case, $H$ is not an $\mathcal{R}^+$-versal unfolding of $h_{\bm{w}_0}$.

\item 
$h_{\bm{w}_0}$ has an $A_{\geq 4}$-singularity at $(0,0)$ if and only if $\bm{w}_0$ is the principal normal vector, $f(\bm{0})$ is not an inflection point and $a_{0,3} = a_{0,4} a_{2,0} - 3 a_{1,2}^2 = 0$. 
When this is the case, $H$ is not an $\mathcal{R}^+$-versal unfolding of $h_{\bm{w}_0}$.

\item 
$h_{\bm{w}_0}$ has a $D_4$- or more degenerate singularity at $(0,0)$ if and only if $\bm{w}_0$ is the principal normal vector and $f(\bm{0})$ is an inflection point.
When this is the case, $H$ is not an $\mathcal{R}^+$-versal unfolding of $h_{\bm{w}_0}$.
\end{enumerate}
\end{proposition}

\begin{proof}
First, we prove the necessary and sufficient condition for $h_{\bm{w}_0}$ having a singularity of type $A_k$ and $D_k$. 
Since $\partial h_{\bm{w}_0}/\partial u = x_0$ and $\partial h_{\bm{w}_0}/\partial v = 0$ at $(0,0)$, $h_{\bm{w}_0}$ is singular if and only if $x_0 = 0$.
If $x_0 = 0$, then
\begin{equation}
\label{eq:2jet_of_h}
h_{\bm{w}_0} = \frac12 a_{2,0} z_0 u^2 + y_0 u v + O(u,v)^3,
\end{equation}
and thus $\det \mathcal{H}_{h_{\bm{w}_0}}(\bm{0}) = - y_0^2$, where $\mathcal{H}_{h_{\bm{w}_0}}$ is the Hessian matrix of $h_{\bm{w}_0}$. 
Hence, (1) is proved.

Assume that $\bm{w}_0 = \pm (0,0,1)$.
Then $h_{\bm{w}_0}$ has a degenerate singularity at $(0,0)$.
It follows from \eqref{eq:2jet_of_h} that $h_{\bm{w}_0}$  has a $D_4$ or more degenerate singularity if and only if $a_{2,0} = 0$.
Assume that $a_{2,0}\ne0$.
By the right coordinate change $\phi(u,v)=(u, v - a_{1,2}v^2/(2a_{2,0}))$ we show that 
\begin{equation*}
h \circ \phi = \pm\left(\dfrac{a_{2,0}}2 u^2 + \dfrac16(a_{3,0}u^2 + 3 a_{2,1} u^2 v + a_{0,3} v^3)\right) + O(u,v)^4
\end{equation*}
and the coefficient of $v^4$ of $h \circ \phi$ is $\pm (a_{0,4}a_{2,0} - 3a_{1,2}^2)/(24a_{2,0})$,  
which proves assertions (2)--(4).

Next, we examine versal unfoldings of $h_{\bm{w}_0}$.
Assume that $h_{\bm{w}_0}$ has an $A_1$-singularity at $(0,0)$.
We may assume that $y\ne0$ near $(u,v,\bm{w})=(0,0,\bm{w}_0)$.
Set $y = \pm \sqrt{1-x^2-z^2}$.
Since $A_1$-singularity is 2-determined, we only have to verify that
\begin{equation}
\label{eq:R_versal_of_h_1}
\mathcal{E}_2 = \left\langle\dfrac{\partial h_{\bm{w}_0}}{\partial u},\dfrac{\partial h_{\bm{w}_0}}{\partial v}\right\rangle_{\mathcal{E}_2} +\left\langle\left.\dfrac{\partial H}{\partial x}\right|_{\R^2\times\{\bm{w}_0\}}, \left.\dfrac{\partial H}{\partial z}\right|_{\R^2\times\{\bm{w}_0\}}\right\rangle_{\R} + \langle 1 \rangle_{\R} + \mathcal{M}_2^3
\end{equation}
holds (see, for example, 
\cite{Martinet1982}), where $\mathcal{E}_2$ is the set of smooth function-germs $(\R^2,\bm{0})\to\R$.
The coefficients of $u^i v^j$ of functions in \eqref{eq:R_versal_of_h_1} are given by the following table:
\[
\begin{array}{c|cc|ccc}
& u & v & u^2 & u v & v^2 \\ \hline
H_x\left|_{\R^2\times\{\bm{w}_0\}}\right. & 1 & 0 & 0 & 0 & 0 \\
H_z\left|_{\R^2\times\{\bm{w}_0\}}\right. & 0 & 0 & a_{2,0} & 0 & 0 \\ \hline
(h_{\bm{w}_0})_u & a_{2,0} z_0 & \pm y_0 & \frac12 a_{3,0} z_0 & a_{2,1} z_0 & \frac12 a_{1,2} z_0 \\
(h_{\bm{w}_0})_v & \pm y_0 & 0 & \frac12 a_{2,1} z_0 & a_{1,2} z_0 & \frac12(a_{0,3} z_0 \pm b_3 y_0) \\ \hline
u (h_{\bm{w}_0})_v & 0 & 0 & \pm y_0 & 0 & 0 \\
u (h_{\bm{w}_0})_u & 0 & 0 & a_{2,0} z_0 & \pm y_0 & 0 \\
v (h_{\bm{w}_0})_u & 0 & 0 & 0 & a_{2,0} z_0 & \pm y_0\\
\end{array}
 \]
Since $y_0 \ne 0$, the matrix represented by the above table is of full rank, that is, \eqref{eq:R_versal_of_h_1} holds.

Assume that $h_{\bm{w}_0}$ has an $A_2$-singularity at $(0,0)$.
We may assume that $z\ne0$ near $(u,v,\bm{w}_0) = (0,0,\bm{w}_0)$.
Set $z = \pm\sqrt{1-x^2-y^2}$.
Since $A_2$-singularity is 3-determined, we have to check the equality
\begin{equation}
\label{eq:R_versal_of_h_2}
\mathcal{E}_2 = \left\langle\dfrac{\partial h_{\bm{w}_0}}{\partial u},\dfrac{\partial h_{\bm{w}_0}}{\partial v} \right\rangle_{\mathcal{E}_2} +\left\langle\left.\dfrac{\partial H}{\partial x}\right|_{\R^2\times\{\bm{w}_0\}}, \left.\dfrac{\partial H}{\partial y}\right|_{\R^2\times\{\bm{w}_0\}}\right\rangle_{\R} + \langle 1 \rangle_{\R} +\mathcal{M}_2^4.
\end{equation}
We have
\begin{gather*}
H_x|_{\R^2 \times \{\bm{w}_0\}} = u, \quad H_y|_{\R^2 \times \{\bm{w}_0\}} = u v,\\
(h_{\bm{w}_0})_u = \pm a_{2,0} u + O(u,v)^2,\quad (h_{\bm{w}_0})_v = \pm\frac12(a_{2,1}u^2 + 2a_{1,2}u v + a_{0,3}v^2)+ O(u,v)^3.
\end{gather*}
Hence, \eqref{eq:R_versal_of_h_2} does not hold, and $H$ is not an $\mathcal{R}^+$-versal unfolding of $h_{\bm{w}_0}$.
For the same reason, $H$ is not an $\mathcal{R}^+$-versal unfolding of $h_{\bm{w}_0}$ having an $A_3$-singularity.

The number of parameters in an $\mathcal{R}^+$-miniversal unfolding of $A_4$ is 3.
Since $H$ is a 2-parameter unfolding of $h_{\bm{w}_0}$, $H$ is not an $\mathcal{R}^+$-versal unfolding of $h_{\bm{w}_0}$ having $A_{\geq 4}$-singularity.
For the same reason, $H$ is not an $\mathcal{R}^+$-versal unfolding of
$h_{\bm{w}_0}$ having a $D_4$ or more degenerate singularity. 
\end{proof}

\begin{remark}
Proposition \ref{prop:height} is applicable more generally to singular surfaces parameterized by corank 1 map-germs whose 2-jets are $\mathcal{A}$-equivalent to $(u,uv,0)$. 
This class includes swallowtails (see \cite{Saji2018}).
\end{remark}

\begin{proof}[Proof of Theorem \ref{thm:height}]
Since the 2-jet of the normal form of $H_k$ is $\mathcal{A}$-equivalent to $(u,uv,0)$, by Proposition \ref{prop:normal_form}, we may assume that $S$ is parameterized by $f$ in the form \eqref{eq:normal_form_2}. 
Moreover, the $3$-jet of $f$ is $\mathcal{A}$-equivalent to $(u,uv,v^3)$, which is the 3-jet of the normal form of $H_k$. 
It follows from the proof of the first assertion of Proposition \ref{prop:H_k} that $a_{0,3}\ne0$.
Therefore, the assertions immediately follow from Proposition \ref{prop:height}. 
\end{proof}

We define the family of functions on $S$ by
$$
\tilde{H}\colon (\R^2 \times S^2 \times \R, (\bm{0},\bm{w}_0,t_0)) \to \R, \quad \tilde{H}(u,v,\bm{w},t)=\langle f(u,v),\bm{w}\rangle - t.
$$
We also define the function $\tilde{h}_{\bm{w},t}(u,v)=\tilde{H}(u,v,\bm{w},t)$, which is called the \emph{extended height function on $S$ along $\bm{w}$}. 

Since $\tilde{H}_t = -1$, from the proof of Proposition \ref{prop:height} and Theorem \ref{thm:height} we have the following corollary:

\begin{corollary}
\label{thm:col}
Let $S$ be a singular surface parameterized by a smooth map-germ $\mathcal{A}$-equivalent to $H_k$.
\begin{enumerate}
\item 
$\tilde{h}_{\bm{w}_0,t_0}$ has an $A_1$-singularity at $(0,0)$ if and only if $\bm{w}_0$ is in the normal plane but $\bm{w}_0$ is not the principal normal direction. 
When this is the case, $\tilde{H}$ is a $\mathcal{K}$-versal unfolding of $\tilde{h}_{\bm{w}_0,t_0}$. 

\item 
$\tilde{h}_{\bm{w}_0,t_0}$ has an $A_2$-singularity at $(0,0)$ if and only if $\bm{w}_0$ is the principal normal direction and the singular point of $S$ is not an inflection point. 
When this is the case, $\tilde{H}$ is not a $\mathcal{K}$-versal unfolding of $\tilde{h}_{\bm{w}_0,t_0}$. 

\item 
$\tilde{h}_{\bm{w}_0,t_0}$ does not have an $A_{\geq 3}$-singularity at $(0,0)$.

\item 
$\tilde{h}_{\bm{w}_0,t_0}$ has a $D_4$- or more degenerate singularity at $(0,0)$ if and only if $\bm{w}_0$ is the principal normal direction and the singular point of $S$ is an inflection point. 
When this is the case, $\tilde{H}$ is not a $\mathcal{K}$-versal unfolding of $\tilde{h}_{\bm{w}_0,t_0}$
\end{enumerate}
\end{corollary}

\section{Local branches of parabolic sets}

In \cite{West1995}, J.~M.~West classified Whitney umbrellas generically into two types in terms of the singularity of the parabolic set in the source. 
A Whitney umbrella whose parabolic set has an $A_1^+$-singularity is classified as a {\it hyperbolic Whitney umbrella}.
A Whitney umbrella whose parabolic set has an $A_1^-$-singularity is classified as an {\it elliptic Whitney umbrella}.
The height function on a hyperbolic Whitney umbrella in any direction in the normal plane has an $A_1$-singularity.
On the other hand, there are two directions in which the height function on an elliptic Whitney umbrella has an $A_{\geq 2}$-singularity.
Each branch of the parabolic set of the elliptic Whitney umbrella is associated with one of the two directions, and the torsion and its derivative of the branch on the elliptic Whitney umbrella relate to the type of the singularity of the height function (\cite[Theorem~2.2]{O-ST2015}). 

Suppose that $S$ is a singular surface parameterized by a smooth map-germ $\mathcal{A}$-equivalent to one of $S_k$, $B_k$, $C_k$ and $F_4$, and that its singular point is not an inflection point. 
T.~Fukui and the author showed that there are two directions in which the height function on $S$ has an $A_{\geq 2}$-singularity.
One of the two directions is the principal normal direction.
The other direction is associated with a special branch of the parabolic set on $S$, and the torsion and its derivative of the branch on $S$ relate to the type of the singularity of the height function (\cite[Theorem 4.2]{FH2025}). 
For comparison, a corresponding relation for regular surfaces is described in \cite[Proposition 4.3]{FH2025}. 

R.~Oset Sinha and F.~Tari determined the generic singularities of the parabolic set in the source of several singular surfaces whose parametrizations have an $\mathcal{A}$-singularity, including those with $H_2$ and $H_3$ (\cite[Theorem~2.7]{O-ST2015}). 
In particular, they showed that the parabolic set generically has a $D_5$-singularity for both $H_2$ and $H_3$. 

Let $S$ be a singular surface parameterized by a smooth map-germ $f\colon(\R^2,\bm{0})\to(\R^3,\bm{0})$ which is $\mathcal{A}$-equivalent to $H_k$. 
Then the parabolic set of $S$ is given by the zero set of 
\[
P(u,v) = (\langle f_u \times f_v, f_{u u} \rangle\langle f_u \times f_v, f_{v v} \rangle - \langle f_u \times f_v, f_{u v} \rangle^2)(u,v).
\]
If $f$ is given in the form \eqref{eq:normal_form_2}, then
\begin{align*}
\begin{split}
j^4 P(\bm{0}) & =  a_{0,3}a_{2,0}u^2 v + a_{1,2}a_{2,0}u^3 + \frac14(4a_{0,3}a_{1,2} + 2a_{2,2}a_{2,0} - a_{2,1}^2)u^4\\
& \quad + \frac12(2a_{1,3}a_{2,0} + 2a_{3,0}a_{0,3} + 2a_{2,1}a_{1,2} - a_{2,1}a_{2,0}b_3)u^3 v\\
& \quad + \frac12(a_{0,4}a_{2,0} + 3a_{2,1}a_{0,3})u^2 v^2 + \frac12a_{0,3}a_{2,0}b_3 u v^3 - \frac14a_{0,3}^2 v^4.
\end{split}
\end{align*}
Since $a_{0,3} \ne 0$, substituting $v$ by $- a_{1,2}u/a_{0,3} + v$, we show that  $j^3 P(\bm{0}) = a_{0,3}a_{2,0}u^2 v$ and the coefficient of $v^4$ of $P$ is $-a_{0,3}^2/4\ne0$. 
It follows that the parabolic set has a $D_5$-singularity at $(0,0)$ if and only if $a_{2,0} \ne0$, namely the singular point of $S$ is not an inflection point. 
When the parabolic set has a $D_5$-singularity, the parabolic set has two local branches: one is a regular curve and the other is a $(2,3)$-cusp.

\begin{theorem}
\label{thm:branch}
Let $S$ be a singular surface parameterized by a smooth map-germ $\mathcal{A}$-equivalent to $H_k$, and suppose that the singular point of $S$ is not an inflection point. 
Let $\gamma(t)$ and $\hat\gamma(t)$ be parameterizations, respectively, of the regular and singular branch of the parabolic set on $S$ with $\gamma(0)$ and $\hat\gamma(0)$ being the singular point. 
Let $\bm{b}(t)$ and $\hat{\bm{b}}(t)$ be the unit binormal vectors of $\gamma$ and $\hat\gamma$, respectively. 
Then the height function on $S$ along $\pm\bm{b}(0)$ has an $A_1$-singularity and that along $\pm\lim_{t \to 0^+}\hat{\bm{b}}(t)$ has an $A_2$-singularity at the singular point of $S$.
\end{theorem}
\begin{proof}
We take the parameterization $f$ of $S$ in \eqref{eq:normal_form_2}.
Remark that $a_{2,0}\ne0$ and $a_{0,3}\ne0$. 
Since the parabolic set has a $D_5$-singularity at $(0,0)$, it has a regular branch and a $(2,3)$-cusp branch. 
We may write their parameterizations in the forms
\begin{equation*}
\alpha(t) = \left(t,\, c_1 t + O(t^2)\right),\quad \hat\alpha(t) = (c_2 t^3 + O(t^4),\, \epsilon t^2)\quad(\epsilon=\pm1),
\end{equation*}
where $O(t^n)$ consists of the terms whose degrees are greater than or equal to $n$, respectively. 
Substituting these parameterizations into $P(u,v)=0$ and comparing the lowest-order terms, we obtain 
$$
a_{2,0}(a_{0,3}c_1+a_{1,2})=0,\qquad \dfrac{a_{0,3}}4(4\epsilon a_{2,0} c_2^2 - a_{0,3})=0,
$$
where $\epsilon = 1$ (resp. $-1$) if $a_{0,3}a_{2,0} > 0$ (resp. $< 0$). 
Hence, 
$$
c_1 = -\dfrac{a_{1,2}}{a_{0,3}},\qquad 4\epsilon a_{2,0} c_2^2 - a_{0,3} = 0.
$$
It follows that we can take $\gamma(t)$ and $\hat\gamma(t)$ in the forms
\begin{align*}
\gamma(t) & = f \circ \alpha(t) = \left(t,\, c_1 t^2 + O(t^3),\, \dfrac{a_{2,0}}2t^2 + O(t^3)\right), \\
\hat\gamma(t) & = f \circ \hat\alpha(t) = \left(c_2 t^3 + O(t^4),\, \epsilon c_2 t^5 + O(t^6),\, \dfrac{7\epsilon a_{0,3}}{24} t^6 + O(t^7)\right).
\end{align*}
Straightforward calculations show that
\[
\bm{b}(0) = \left(0,\, -\dfrac{a_{2,0}}{\sqrt{4c_1^2 + a_{2,0}^2}},\, \dfrac{2c_1}{\sqrt{4c_1^2 + a_{2,0}^2}}\right)\ne \pm(0,0,1)
\]
and
\[
\hat{\bm{b}}(t) =\dfrac{\left(\frac{35}4 c_2 a_{0,3}t^8 + O(t^9),\, -\frac{63}4 \epsilon c_2 a_{0,3} t^6 + O(t^7),\, 30 \epsilon c_2^2 t^5 + O(t^6)\right)}{30 c_2^2 |t^5| \sqrt{1 + O(1)}},
 \]
and thus $\lim_{t\to 0^+}\hat{\bm{b}}(t) = (0,0,\epsilon)$.
From (1) and (2) of Proposition \ref{prop:height}, we complete the proof. 
\end{proof}

\end{document}